\documentclass[11pt,twoside,reqno]{amsart}
\usepackage[latin1]{inputenc}
\usepackage[OT1]{fontenc}
\usepackage{amsmath}
\usepackage{amsthm}
\usepackage{amssymb}
\usepackage[all]{xy}
\usepackage{bbm}
\usepackage{amscd}
\usepackage{a4wide}
\usepackage{enumitem}
\usepackage{mathtools}

\setenumerate[0]{label=(\roman*)}

\DeclareMathOperator{\genus}{\mathsf{g}}
\DeclareMathOperator{\NN}{\mathsf{N}}

\DeclareMathOperator{\cc}{\mathsf{c}}

\DeclareMathOperator{\Spec}{\mathsf{Spec}}

\DeclareMathOperator{\pr}{\mathsf{pr}}

\DeclareMathOperator{\id}{\mathsf{id}}

\DeclareMathOperator{\Hom}{\mathsf{Hom}}

\DeclareMathOperator{\Coh}{\mathsf{Coh}}

\DeclareMathOperator{\Ho}{\mathsf H}
\DeclareMathOperator{\For}{\mathsf{For}}

\DeclareMathOperator{\rank}{\mathsf{rank}}

\let\deg\relax
\DeclareMathOperator{\deg}{\mathsf{deg}}

\DeclareMathOperator{\Aut}{\mathsf{Aut}}

\DeclareMathOperator{\Pic}{\mathsf{Pic}}

\newcommand{\cI}{{\mathcal I}}
\newcommand{\cJ}{{\mathcal J}}

\newcommand{\IC}{\mathbb{C}}

\newcommand{\IN}{\mathbb{N}}
\newcommand{\IP}{\mathbb{P}}

\newcommand{\wH}{\widetilde H}

\DeclareMathOperator{\VB}{\mathsf{VB}}

\DeclareMathOperator{\TT}{\mathsf{T}\!}

\DeclareMathOperator{\FF}{\mathsf{F}}

\DeclareMathOperator{\CC}{\mathsf{C}}

\DeclareMathOperator{\Hi}{\mathsf{Hi}}

\makeatletter
\newcommand{\leqnomode}{\tagsleft@true}
\newcommand{\reqnomode}{\tagsleft@false}
\makeatother

\let\ker\relax
\DeclareMathOperator{\ker}{\mathsf{ker}}

\let\dim\relax
\DeclareMathOperator{\dim}{\mathsf{dim}}

\let\log\relax
\DeclareMathOperator{\log}{\mathsf{log}}

\newcommand{\sym}{\mathfrak S}

\newcommand{\cP}{\mathcal P}
\newcommand{\cQ}{\mathcal Q}

\newcommand{\reg}{\mathcal O}

\newcommand{\eps}{\varepsilon}

\renewcommand{\theta}{\vartheta}
\renewcommand{\rho}{\varrho}
\renewcommand{\phi}{\varphi}
\renewcommand{\_}{\underline{\,\,\,\,}}

\usepackage[usenames,dvipsnames]{xcolor}
\usepackage{hyperref}
\hypersetup{
    colorlinks,
    linkcolor={red!50!black},
    citecolor={blue!50!black},
    urlcolor={blue!80!black}
}
\usepackage{aliascnt} % to get hyperref's \autoref to work properly with counters

\newtheorem{theorem}{Theorem}[section]

 \newaliascnt{proposition}{theorem}
  \newtheorem{prop}[proposition]{Proposition}
  \aliascntresetthe{proposition}

  \newaliascnt{lemma}{theorem}
  \newtheorem{lemma}[lemma]{Lemma}
  \aliascntresetthe{lemma}

  \newaliascnt{corollary}{theorem}
  \newtheorem{cor}[corollary]{Corollary}
  \aliascntresetthe{corollary}

\theoremstyle{definition}
  \newaliascnt{definition}{theorem}
  
  \aliascntresetthe{definition}

  \newaliascnt{remark}{theorem}
  \newtheorem{remark}[remark]{Remark}
  \aliascntresetthe{remark}

  \newaliascnt{condition}{theorem}
  
  \aliascntresetthe{condition}

  \newaliascnt{question}{theorem}
  
  \aliascntresetthe{question}

  \newaliascnt{example}{theorem}
  \newtheorem{example}[example]{Example}
  \aliascntresetthe{example}

\begin{document}

\title[Logarithmic Higgs Bundles on punctual Hilbert schemes]{Fourier-Mukai transformation and
Logarithmic Higgs Bundles on 
punctual Hilbert schemes}

\author[I.\ Biswas]{Indranil Biswas}

\address{School of Mathematics, Tata Institute of Fundamental
Research, Homi Bhabha Road, Mumbai 400005, India}

\email{indranil@math.tifr.res.in}

\author[A.\ Krug]{Andreas Krug}

\address{Fachbereich Mathematik und Informatik,
Philipps-Universit\"at Marburg, Lahnberge, Hans-Meerwein-Strasse, D-35032
Marburg, Germany}

\email{andkrug@mathematik.uni-marburg.de}



\begin{abstract}
Given a vector bundle $E$ on a smooth projective curve or surface $X$ carrying the structure 
of a $V$-twisted Hitchin pair for some vector bundle $V$, we observe that the associated 
tautological bundle $E^{[n]}$ on the punctual Hilbert scheme of points $X^{[n]}$ has an 
induced structure of a $((V^\vee)^{[n]})^\vee$-twisted Hitchin pair, where $(V^\vee)^{[n]}$ 
is a vector bundle on $X^{[n]}$ constructed using the dual $V^\vee$ of $V$. In particular, a 
Higgs bundle on $X$ induces a logarithmic Higgs bundle on the Hilbert scheme $X^{[n]}$. We 
then show that the known results on stability of tautological bundles and reconstruction from 
tautological bundles generalize to tautological Hitchin pairs.
\end{abstract}

\maketitle

\section{Introduction}

Let $X$ be a smooth projective variety with $\dim X\le 2$. In this case, the Hilbert scheme 
$X^{[n]}$ of $n$ points on $X$ is a smooth projective variety of dimension $\dim 
X^{[n]}=n\cdot\dim X$. If $X$ is a curve, $X^{[n]}$ coincides with the symmetric product 
$X^{(n)}:=X^n/\sym_n$, where $\sym_n$ is the group of permutations of $\{1,\dots ,n\}$. If 
$X$ is a surface, then $X^{[n]}$ is a resolution of the singularities of $X^{(n)}$ via the 
Hilbert--Chow morphism $\mu\colon X^{[n]}\to X^{(n)}$. The points parametrizing non-reduced 
sub-schemes of $X$ form a divisor $B_n\subset X^{[n]}$. Given a vector bundle $E$ on $X$, we 
get a vector bundle $E^{[n]}$ on $X^{[n]}$ which is the direct image of the pullback of $E$ 
to the universal subscheme $\Xi_n\subset X\times X^{[n]}$. These vector bundles on 
$X^{[n]}$ play a crucial role in the investigation of the topology and geometry of the Hilbert 
schemes \cite{LehnChern, LehnSorger1, LehnSorger2}, but are also useful tools for studying properties of the bundles on $X$ itself; see, for example \cite{Voisin-syzygy, Ein-Lazarsfeld--gonalityconj, Agostini}.

Higgs bundles constitute an extensively studied topic; they appear in algebraic geometry,
differential geometry, symplectic geometry and also in representation theory \cite{Ngo}. We recall that the
Higgs bundles on Riemann surfaces were introduced by Hitchin \cite{Hi} and the Higgs
bundles on higher dimensional complex manifolds were introduced by Simpson \cite{Si}.

It turns out that a Higgs field $E\to E\otimes \Omega_X$ of a vector bundle $E$ on $X$ does not, in general, 
induce a Higgs field of $E^{[n]}$; see \autoref{ex:noHiggs}. However we observe that it induces a 
\emph{logarithmic Higgs field} on $E^{[n]}$, that is a homomorphism $E^{[n]}\to E^{[n]}\otimes 
\Omega_{X^{[n]}}(\log B_n)$ or, equivalently, a homomorphism $E^{[n]}\otimes \TT_{X^{[n]}}(-\log B_n)\to 
E^{[n]}$.

The construction is as follows. As proved by Stapleton \cite[Thm.\ B]{Stapletontaut}, there is an isomorphism $\TT_{X^{[n]}}(-\log B_n)\cong (\TT_X)^{[n]}$, identifying the sheaf of logarithmic vector fields on $X^{[n]}$ with the tautological bundle associated to the tangent sheaf of $X$. Now, the logarithmic Higgs field on $E^{[n]}$ associated to a Higgs field $\theta\colon E\otimes \TT_X\to E$ is given by the composition
\[
E^{[n]}\otimes \TT_{X^{[n]}}(-\log B_n)\cong E^{[n]}\otimes \TT_X^{[n]}\xrightarrow\nu (E\otimes \TT_X)^{[n]}\xrightarrow{\theta^{[n]}} E^{[n]}\, 
\]
where $\nu \colon E^{[n]}\otimes \TT_X^{[n]}\to (E\otimes \TT_X)^{[n]}$ is the cup product relative to the projection morphism $\pr_{X^[n]}\colon \Xi_n\to X^{[n]}$; compare \autoref{subsect:push} and \autoref{rem:taut}.
We denote this induced Higgs field by $\theta^{\{n\}}\colon E^{[n]}\otimes \TT_{X^{[n]}}(-\log B_n)\to E^{[n]}$, and get the \emph{tautological logarithmic Higgs bundle} $(E,\theta)^{[n]}:=(E,\theta)$.

There is the more general notion of a \emph{$V$-cotwisted Hitchin pair}, for $V$ a vector bundle of $X$, such 
that a Higgs bundle on $X$ is exactly a $\TT_X$-cotwisted Hitchin pair, and a logarithmic Higgs bundle on 
$X^{[n]}$ is exactly a $\TT_{X^{[n]}}(-\log B_n)$-cotwisted Hitchin pair; see \autoref{subsect:co-pairs}. The 
above construction generalizes in such a way that for a $V$-cotwisted Hitchin pair $(E,\theta)$ on $X$, we get 
a tautological $V^{[n]}$-cotwisted Hitchin pair $(E,\theta)^{[n]}=(E^{[n]},\theta^{\{n\}})$ on $X^{[n]}$; see 
\autoref{subsect:taut}.

In this article, we proof that basically all the known results on stability of tautological bundles and 
reconstruction from tautological bundles lift to results for tautological Hitchin pairs. As probably the most 
interesting special case, we get results on tautological logarithmic Higgs bundles. Concretely, we proof the 
following results.

\begin{theorem}\label{thm:g2}
Let $C$ be a smooth projective curve of genus $\genus(C)\ge 2$, and let $(E,\theta)$, $(F,\eta)$ be semi-stable $V$-cotwisted Hitchin bundles for some vector bundle $V\in \VB(C)$. Then, for every $n\in \IN$, we have
\[
 (E,\theta)^{[n]}\cong (F,\eta)^{[n]}\quad\implies \quad (E,\theta)\cong (F,\eta)\,.
\]
\end{theorem}

For a scheme $Y$, by $\VB(Y)$ and $\Coh(Y)$ we denote the category of vector bundles and coherent sheaves
respectively on $Y$.

\begin{theorem}\label{thm:g1}
Let $C$ be an elliptic curve, and let $(E,\theta)$, $(F,\eta)$ be $V$-cotwisted Hitchin bundles for some vector bundle $V\in \VB(C)$. Then, for every $n\in \IN$, we have
\[
 (E,\theta)^{[n]}\cong (F,\eta)^{[n]}\quad\implies \quad (E,\theta)\cong (F,\eta)\,.
\]
\end{theorem}

\begin{theorem}\label{thm:rechigher}
 Let $X$ be a smooth quasi-projective variety of dimension $\dim X\ge 2$, let $V\in\VB(X)$, and let $(E,\theta),(F,\eta)$ be $V$-cotwisted Hitchin pairs for some vector bundle $V\in \VB(C)$, such that the
coherent sheaves $E$ and $F$ are reflexive.
Then, for every $n\in \IN$,
\[
 (E,\theta)^{[n]}\cong (F,\eta)^{[n]}\quad\implies \quad (E,\theta)\cong (F,\eta)\,.
\]
\end{theorem}

Let $C$ be a smooth projective curve. In the group $\NN^1(C^{(n)})$ of divisors modulo numerical equivalence, we consider the ample class $H_n$ that is represented by $C^{(n-1)}+x\subset C^{(n)}$ for any $x\in C$; see \cite[Sect.\ 1.3]{Krug--stab} for details.

\begin{theorem}\label{thm:curvestab}
Let $C$ be a smooth projective curve, and let $(E,\theta)$ be a Hitchin bundle on $C$.
Then for every $n\in \IN$, the following two statements hold.
\begin{enumerate}
 \item If $(E,\theta)$ is stable and $\mu(E)\notin[-1,n-1]$, then the logarithmic Higgs bundle $(E^{[n]},\theta^{\{n\}})$ is stable with respect to $H_n$.
 \item If $(E,\theta)$ is semi-stable and $\mu(E)\notin(-1,n-1)$, then the logarithmic Higgs bundle $(E^{[n]},\theta^{\{n\}})$ is semi-stable with respect to $H_n$.
\end{enumerate}
\end{theorem}

The paper is organized as follows. In \autoref{sect:Hitchin}, we collect basic notions and results on Hitchin 
pairs needed for later use. In particular, we define the push-forward of Hitchin pairs along flat and finite 
morphisms, and define Fourier--Mukai transforms of Hitchin pairs for a certain class of kernels; see 
\autoref{subsect:push} and \autoref{subsect:FM}. In \autoref{sect:Hilbert}, we explain, in some more detail 
than done in this introduction, how to get induced structures of Hitchin pairs and logarithmic Higgs bundles 
on tautological bundles on Hilbert schemes of points. All four of our results listed above have already been 
proved in the special case of ordinary sheaves without the structure of a Hitchin pair.
(A coherent sheaf is the same as a $0$-cotwisted Hitchin pair.) For \autoref{thm:g2}, see 
\cite{Biswas-Nagaraj--reconstructioncurves} and \cite[Sect.\ 2]{Biswas-Nagaraj--reconstructionsurfaces}; for 
\autoref{thm:g1}, see \cite[Thm.\ 1.5]{Krug--reconstruction}; for \autoref{thm:rechigher}, see \cite[Thm.\ 3.2 
\& Rem.\ 3.5]{Krug--reconstruction}; for \autoref{thm:curvestab}, see \cite{Krug--stab}. In all cases, the 
proofs can be amended in such a way that they work for Hitchin pairs. We follow different approaches in 
explaining how to amend the proofs. For \autoref{thm:g2} and \autoref{thm:g1}, we give full proofs in 
\autoref{sect:reccurves}. The reader should be able to follow these proofs without going back to 
\cite{Biswas-Nagaraj--reconstructioncurves}, \cite{Biswas-Nagaraj--reconstructionsurfaces} or 
\cite{Krug--reconstruction}, though some steps of the arguments might be carried out in greater details in 
these articles. In contrast, for the proofs of \autoref{thm:rechigher} and \autoref{thm:curvestab}, we only 
explain the extra ingredients needed to lift the proofs from the cases of ordinary sheaves to Hitchin pairs, 
and where to insert these ingredients. Hence, for following these proofs, the reader should at the same time 
have a look at the relevant parts of \cite{Krug--reconstruction} and \cite{Krug--stab}.

In the final \autoref{subsect:surfacestab}, we remark that also a result of Stapleton \cite{Stapletontaut} on stability of tautological bundles on Hilbert schemes of points on surfaces extends to Hitchin bundles.

\section{Preliminaries on Hitchin pairs}\label{sect:Hitchin}

Throughout this section, let $X$ be a scheme over the complex numbers $\IC$. Usually, we will work with 
varieties, but at one point in \autoref{subsect:g2} we will have to consider Hitchin pairs on infinitesimal 
neighborhoods of a diagonal, which is why we choose the greater generality here.

\subsection{Basic Definitions}\label{subsect:co-pairs}

Let $V\in \VB(X)$ be a vector bundle. A \emph{$V$-cotwisted Hitchin pair on $X$} is a pair
$(E,\theta)$ consisting of a coherent sheaf $E\in \Coh(X)$ and an $\reg_X$-linear map $\theta\colon E\otimes V\to V$ such that the composition 
\[
\bigl(\theta\circ (\theta\otimes \id_V)\bigr)_{\mid \bigwedge\nolimits^2V}\colon E\otimes\bigwedge\nolimits^2V\hookrightarrow E\otimes V\otimes V\xrightarrow{\theta\otimes \id_V} E\otimes V\xrightarrow{\theta}E\,,
\]
where $\bigwedge^2V\subset V\otimes V$ is the second exterior product, vanishes. A Hitchin pair
$(E,\theta)$ where $E$ is a vector bundle is also called a \emph{Hitchin bundle}.

A \emph{morphism} between two $V$-cotwisted Hitchin pairs $(E,\theta)$ and $(F,\eta)$ on $X$ is an $\reg_X$-linear map $\gamma\colon E\to F$ such that the following diagram commutes
\[
 \xymatrix{
 E\otimes V\ar[d]_{\gamma\otimes \id_V} \ar[r]^{\theta} & E\ar[d]^{\gamma} \\
F\otimes V \ar[r]^{\eta} & F \,.
 }
\]
We denote the category of $V$-cotwisted Hitchin pairs on $X$ by $\Hi^V(X)$. It is an abelian category with 
kernels and cokernels given by the kernels and cokernels of the underlying morphisms of coherent sheaves,
equipped with the induced $\reg_X$-linear homomorphism. In 
other words, the forgetful functor $$\For\colon \Hi^V(X)\to \Coh(X)$$ is exact.

\begin{remark}\label{rem:cotwisttwist}
In the literature, usually \emph{$V$-twisted} Hitchin pairs are considered instead of $V$-cotwisted Hitchin pairs. A $V$-twisted Hitchin pair is a pair $(E,\zeta)$ consisting of $E\in \Coh(X)$ and an $\reg_X$-linear morphism $\zeta\colon E\to E\otimes V$ such that the composition 
\[
 E\xrightarrow{\zeta}E\otimes V\xrightarrow{\zeta\otimes \id_V} E\otimes V\otimes V\twoheadrightarrow E\otimes
\bigwedge\nolimits^2V
\]
vanishes. However, the category $\Hi_{V^\vee}(X)$ of $V^\vee$-twisted Hitchin pairs is equivalent to $\Hi^V(X)$ via the pair of mutually inverse exact functors
\begin{align*}
 &\Hi^V(X)\to \Hi_{V^\vee}(X)\,,\quad(E,\theta)\mapsto (E,\theta')\,,\quad \theta':=\bigl(E\xrightarrow{\id_E\otimes\eta_V}E\otimes V\otimes V^\vee\xrightarrow{\theta\otimes \id_{V^\vee}} E\otimes V^\vee\bigr)\,,\\
&\Hi_{V^\vee}(X)\to \Hi^V(X)\,,\quad(E,\zeta)\mapsto (E,\zeta')\,,\quad \zeta':=\bigl(E\otimes V\xrightarrow{\zeta\otimes \id_V}E\otimes V^\vee\otimes V\xrightarrow{\id_{E}\otimes \eps_V} E\bigr)\,,
\end{align*}
where $\eta_V\colon \reg_X\to V\otimes V^\vee$ sends any $f$ to
$f\cdot\id_V$, and $\eps_V\colon V^\vee\otimes V\to \reg_X$ is the trace map.
The reason that we prefer to work with $V$-cotwisted sheaves is that they allow an easier formulation of a push-forward functor (under certain circumstances); see \autoref{rem:copush}.

\end{remark}

\subsection{Hitchin subsheaves}

Let $(E,\theta)$ be a $V$-cotwisted Hitchin pair on $X$. A coherent subsheaf $A\subset E$ is called a 
\emph{Hitchin subsheaf} if $\theta(A\otimes V)\subset A$. A Hitchin subsheaf carries itself the structure of 
a $V$-cotwisted Hitchin pair $(A,\theta_{\mid A})$ which makes the inclusion $A\hookrightarrow E$ a morphism 
of Hitchin pairs. Furthermore, the quotient $E/A$ carries the structure of a Hitchin pair $(E/A,\bar\theta)$ 
such that the quotient map $E\twoheadrightarrow E/A$ is a morphism of Hitchin pairs.

For later use, we note the following quite obvious statements.
\begin{lemma}\label{lem:subsheaves}
Let $(E,\theta)$ be a Hitchin pair, and let $A\subset E$ be a Hitchin subsheaf.
\begin{enumerate}
 \item If $B\subset E$ is another Hitchin subsheaf, then the intersection $A\cap B\subset E$ is a Hitchin subsheaf too.
 \item Let $(F,\eta)$ be another Hitchin pair, and let $\phi\colon E\to F$ be a morphism of Hitchin pairs. Then $\phi(A)\subset F$ is a Hitchin subsheaf. 
\end{enumerate}
\end{lemma}

\subsection{Higgs bundles}

Let $X$ be a smooth variety with tangent bundle $\TT_X$. A \emph{Higgs bundle} on $X$ is a $\TT_X$-cotwisted Hitchin pair $(E,\theta)$.

Given a reduced divisor $D\subset X$, we consider the \emph{sheaf of logarithmic vector fields} $\TT_X(-\log 
D)$. It is defined as the sheaf of vector fields which along the regular locus $D_{\mathsf{reg}}$ of $D$ are 
tangent to $D$. A \emph{logarithmic Higgs bundle} (with respect to the divisor $D$) is a $\TT_X(-\log 
D)$-cotwisted Hitchin pair $(E,\theta)$.

Logarithmic Higgs bundles on curves were first considered by Bottacin \cite{Bo}. He proved that a
moduli space of logarithmic Higgs bundle on a curve has a natural Poisson structure.

We note that a logarithmic Higgs bundle is a parabolic Higgs bundle with trivial quasi-parabolic structure. Mochizuki has proved many important results on parabolic Higgs bundles \cite{Mo1}, \cite{Mo2}.

\subsection{Change of the cotwisting bundle}\label{subsect:change}

Let $\phi\colon V\to W$ be a morphism of vector bundles on $X$. There is an induced exact functor 
\[
\phi^\#\colon \Hi^W(X)\to \Hi^V(X)\quad,\quad(E,\theta)\mapsto \bigl(E,\theta\circ (\id_E\otimes \phi)\bigr)\,. 
\]

\begin{lemma}\label{lem:induced}
Let $\phi\colon V\to Q$ be a surjective morphism of vector bundles with kernel $K:=\ker(p)$.
\begin{enumerate}
 \item The functor $\phi^\#\colon \Hi^Q(X)\to \Hi^V(X)$ is fully faithful.
 \item A Hitchin pair $(E,\theta)\in \Hi^V(X)$ is in the essential image of the functor $\phi^\#$ if and only if $\theta(E\otimes K)=0$.
 \item Let $(E,\theta)\in \Hi^Q(X)$. A subsheaf $A\subset E$ is a Hitchin subsheaf of $(E,\theta)$ if and only if it is a Hitchin subsheaf of $\phi^\#(E,\theta)$.
\end{enumerate}
\end{lemma}

\begin{proof} Let $(E,\theta), (F,\eta)\in \Hi^Q(X)$. By the surjectivity of $\phi$, a $\reg_X$-linear 
morphism $\gamma\colon E\to F$ is a morphism between the $Q$-cotwisted Hitchin pairs $(E,\theta)$ and 
$(F,\eta)$ if and only if it is a morphism between the $V$-cotwisted Hitchin pairs $\phi^\#(E,\theta)$ and 
$\phi^\#(F,\eta)$, which proves (i). Part (ii) is straight-forward to prove.
Part (iii) 
follows directly from the surjectivity of $\phi$ and the definition of Hitchin subsheaves. \end{proof}

\subsection{Pull-back of Hitchin pairs}\label{subsect:pull}

Let $f\colon X'\to X$ be a morphism, and let $E$ and $V$ be coherent sheaves on $X$. There is a natural isomorphism
\[
\alpha:=\alpha_{f,E,V}\colon f^*E\otimes f^*V\xrightarrow\sim f^*(E\otimes V)\,. 
\]
Concretely, if $f\colon \Spec A'\to \Spec A$ is a morphism of affine schemes, so that $E=\widetilde M$, $V=\widetilde N$ for some $A$-modules $M$ and $N$, then $\alpha$ is given by the map
\begin{equation}\label{eq:alpha}
 (M\otimes_A A')\otimes_{A'}(N\otimes_A A')\to (M\otimes_A N)\otimes_A A'\quad,\quad (m\otimes a)\otimes (n\otimes b)\mapsto (m\otimes n)\otimes (a b)\,.
\end{equation}
Using the isomorphism $\alpha$, we can define the pull-back along $f$ on the level of Hitchin pairs as 
\[
 f^*\colon \Hi^V(X)\to \Hi^{f^*V}(X')\quad,\quad (E,\theta)\mapsto (f^*E, f^*\theta\circ \alpha)\,.
\]
We often omit the isomorphism $\alpha$ in our notation and simply write $f^*\theta\colon f^*E\otimes f^*V\to f^*E$. 

\subsection{Hitchin pairs under tensor products}\label{subsect:tensor}

Let $(E,\theta)\in \Hi^V(X)$ be a Hitchin pair, and let $P\in \Coh(X)$. We write $P\otimes(E,\theta)$ for the Hitchin pair $(P\otimes E,\id_P\otimes \theta)$. Clearly, this defines a functor
$ P\otimes \_\colon \Hi^V(X)\to \Hi^V(X)$.

\subsection{Push-forwards of Hitchin pairs}\label{subsect:push}

Let now $\pi\colon X\to Y$ be a morphism such that $\pi_*V$ is again a vector bundle. Note that this is always the case if $\pi$ is flat and finite. In the following, for every $(E,\theta)\in \Hi^V(X)$, we will naturally equip $\pi_*E$ with the structure of a $\pi_*V$-cotwisted Hitchin pair. There is a natural morphism 
\[
 \nu:=\nu_{\pi,E,V}\colon \pi_*E\otimes \pi_*V\to \pi_*(E\otimes V)\,.
\]
Over an open affine subset $\Spec A=U\subset X$, it is given by the $A$-linear cup product
\begin{align}\label{eq:nu}
 \Gamma(\pi^{-1}U, E)\otimes_A \Gamma(\pi^{-1} U, V)\xrightarrow\cup \Gamma(\pi^{-1} U , E\otimes V)\,. 
\end{align}

\begin{lemma}\label{lem:push}
Let $\pi\colon X\to Y$ be a morphism of schemes, let $(E,\theta)\in \Hi^V(X)$, and set 
\[
 \check \theta:=\pi_*\theta\circ \nu\colon \pi_*E\otimes \pi_*V\to \pi_*E\,.
\]
Then the following vanishing statement holds:
$$\bigl(\check\theta\circ(\check\theta\otimes \id_{\pi_*V})\bigr)_{\mid\pi_*E\otimes \bigwedge\nolimits^2 (\pi_*V)}=0\, .$$
\end{lemma}

\begin{proof}
From the description \eqref{eq:nu} of $\nu$ over open affine subsets, we see that the diagram
\[
 \xymatrix{
 \pi_*(E\otimes V)\otimes \pi_*V\ar[r]^{\quad\nu_{E\otimes V,V}} \ar[d]_{\pi_*\theta\otimes \id_{\pi_*V}} & \pi_*(E\otimes V\otimes V) \ar[d]^{\pi_*(\theta\otimes \id_V)} \\
 \pi_*E\otimes \pi_*V\ar[r]^{\nu_{E,V}} & \pi_*(E\otimes V) 
 }
\]
commutes. Hence, we can rewrite the composition $\check\theta\circ(\check\theta\otimes \id_V)$ as
\begin{align}
 \check\theta\circ(\check\theta\otimes \id_V)&=(\pi_*\theta)\circ \nu_{E,V}\circ (\pi_*\theta\otimes\id_{\pi_*V})\circ(\nu_{E,V}\otimes \id_{\pi_*V})\notag\\&=(\pi_*\theta)\circ (\pi_*(\theta\otimes\id_V))\circ\nu_{E\otimes V,V}\circ(\nu_{E,V}\otimes \id_{\pi_*V})\notag
 \\&=(\pi_*(\theta\circ (\theta\otimes\id_V)))\circ\nu_{E,V\otimes V}\circ(\id_{\pi_*E}\otimes \nu_{V,V})\label{eq:push1}
 \,.
\end{align}
It follows from \eqref{eq:nu} that $\nu_{V,V}(\bigwedge^2 (\pi_*V))\subset \pi_*(\bigwedge^2V)$, hence
\begin{equation}\label{eq:push2}
\nu_{E,V\otimes V}(\id_{\pi_*E}\otimes \nu_{V,V})(\pi_*E\otimes \bigwedge\nolimits^2(\pi_*V))\subset \pi_*(E\otimes
\bigwedge\nolimits^2V)\,. 
\end{equation}
Since $(E,\theta)$ is a Hitchin pair, we have that $(\theta\circ(\theta\otimes\id_V))_{\mid E\otimes \bigwedge^2V}=0$, hence
\begin{equation}\label{eq:push3}
 (\pi_*(\theta\circ (\theta\otimes\id_V)))_{\mid \pi_*(E\otimes \bigwedge\nolimits^2V)}=0\,.
\end{equation}
The combination of \eqref{eq:push1}, \eqref{eq:push2}, and \eqref{eq:push3} gives the desired vanishing.
\end{proof}

\begin{cor}\label{cor:push}
 Let $\pi\colon X\to Y$ be a morphism of schemes, and let $V\in \VB(X)$ be a vector bundle with the property that $\pi_*V$ is again a vector bundle. Then there is a push-forward functor
 \[
\pi_*\colon \Hi^V(X)\to \Hi^{\pi_*V}(Y)\quad,\quad \pi_*(E,\theta)=(\pi_*E,\check\theta)\,.
 \]
In particular, if $\pi\colon X\to Y$ is flat and finite, there exists for every vector bundle $V\in \VB(X)$ a push-forward functor $\pi_*\colon \Hi^V(X)\to \Hi^{\pi_*V}(Y)$ . 
\end{cor}

\begin{remark}\label{rem:copush}
\autoref{cor:push} is the reason that we are working with cotwisted instead of twisted Hitchin pairs; compare \autoref{rem:cotwisttwist}. Indeed, in terms of twisted Hitchin pairs, the push-forward is a functor $\Hi_V(X)\to \Hi_{(\pi_*(V^\vee))^\vee}(Y)$, which makes it a bit inconvenient to formulate things in terms of twisted Hitchin pairs whenever a push-forward occurs. 
\end{remark}

Given a Cartesian diagram of schemes
\begin{equation}\label{diag:cart}
 \xymatrix{
 X'\ar[r]^{f'} \ar[d]_{\pi'} & X \ar[d]^{\pi} \\
 Y'\ar[r]^{f} & Y 
 }
\end{equation}
and a coherent sheaf $E\in \Coh(X)$, we denote the natural base change morphism by 
\[
\beta:=\beta_{E}\colon f^*\pi_*E\to \pi'_*f'^*E\,. 
\]
In the case that all the schemes involved are affine, which means that \eqref{diag:cart} is the spectrum of a diagram of commutative rings of the form 
\begin{equation}\label{diag:cartaffine}
 \xymatrix{
 B' & \ar[l] B \\
 A' \ar[u]_{\phi'} & \ar[l] A \ar[u]^{\phi} \,, 
 }
\end{equation}
with $E=\widetilde M$ and $V=\widetilde N$ for some $B$-modules $M$ and $N$, the map $\beta$ is given by
\begin{equation}\label{eq:beta}
 (M_A)\otimes_A A'\to (M\otimes_B B')_{A'} \quad,\quad m\otimes a\mapsto m\otimes \phi'(a)\,.
\end{equation}

\begin{lemma}\label{lem:basechange}
Let $\pi$ be flat and finite, and let $(E,\theta)\in \Hi^V(X)$. Then 
\[
 \beta_E\colon f^*\pi_*(E,\theta)\xrightarrow\sim \beta_V^\# \pi'_*f'^*(E,\theta)
\]
is an isomorphisms of $f^*\pi_*V$-cotwisted Hitchin pairs on $Y'$.
\end{lemma}

\begin{proof}
By the assumptions on $\pi$, the map $\beta_{E}\colon f^*\pi_*E\to \pi'_*f'^*E$ is an isomorphism of coherent sheaves. It remains to proof that it is a morphism of Hitchin pairs, which amounts to checking the commutativity of the diagram
\[
 \xymatrix{
f^*\pi_* E\otimes f^*\pi_*V \ar[rr]^{\alpha} \ar[d]_{\beta_E\otimes \id} & & f^*(\pi_* E\otimes \pi_*V) \ar[r]^{f^*\nu} & f^*\pi_*( E\otimes V)\ar[r]^{f^*\pi_*\theta} & f^*\pi_* E \ar[d]^{\beta_E} \\
 \pi'_*f'^*E\otimes f^*\pi_* V \ar[r]^{\id\otimes \beta_V} & \pi'_*f'^*E\otimes \pi'_*f'^*V \ar[r]^{\nu} & \pi'_*(f'^*E\otimes f'^*V) \ar[r]^{\pi'_*\alpha} & \pi'_*f'^*(E\otimes V) \ar[r]^{\pi'_*f'^*\theta} & \pi'_*f'^*E \,.
 }
\]
For this, we may assume that \eqref{diag:cart} is given by the spectrum of \eqref{diag:cartaffine}, 
$E=\widetilde M$, and $V=\widetilde N$. Then, using the concrete descriptions \eqref{eq:alpha}, \eqref{eq:nu}, and \eqref{eq:beta} of the maps $\alpha$, $\nu$, and $\beta$,
it can be checked that both paths of the above diagram are given by 
\[
 (M_A\otimes_A A')\otimes_{A'}(N_A\otimes_A A')\to (M\otimes_BB')_{A'}\quad,\quad (m\otimes a_1)\otimes
(n\otimes a_2)\mapsto \theta(m\otimes n)\otimes \phi'(a_1a_2)\,.\qedhere
\]
\end{proof}

Later, we usually omit the equivalence $\beta^\#\colon \Hi^{f^*\pi_*V}(Y')\xrightarrow \sim \Hi^{\pi_*'f'^*V}(Y')$ when applying \autoref{lem:basechange}, and simple write $f^*\pi_*(E,\theta)\cong \pi'_*f'^*(E,\theta)$.

\subsection{Fourier--Mukai transforms of Hitchin pairs}\label{subsect:FM}

Let $X$ and $Y$ be varieties, and let $Z\subset X\times Y$ be a closed sub-scheme. We denote by $\pr_X\colon X\times Y\to X$ and $\pr_Y\colon X\times Y\to Y$ the projections to the factors, and write $p\colon Z\to X$ and $q\colon Z\to Y$ for the restrictions of these projections to $Z$. In the following, we will always assume that $q\colon Z\to Y$ is flat and finite. For a coherent sheaf $\cP\in \Coh(Z)$, we define the \emph{Hitchin enhanced Fourier--Mukai transform}
$\Phi^{Z}_{\cP}\colon \Hi^V(X)\to \Hi^{q_*p^*V}(Y)$ as the following composition of the three functors discussed in Subsections \ref{subsect:pull}, \ref{subsect:tensor}, and \ref{subsect:push}:
\[
 \Phi^{Z}_{\cP}\colon \Hi^V(X)\xrightarrow{p^*}\Hi^{p^*V}(Z)\xrightarrow{\cP\otimes}\Hi^{p^*V}(Z)\xrightarrow{q_*}\Hi^{q_*p^*V}(Y)\,.
\]

\begin{remark}
Note that, in contrast to the usual convention for Fourier--Mukai transforms, none of our functors are derived. 
Hence, the functor $\Phi^{Z}_{\cP}\colon \Hi^V(X)\to\Hi^{q_*p^*V}(Y)$ needs not to be exact. However, it is exact whenever $\cP$ is flat over $X$, as will be the case in all our applications; compare \cite[Thm.\ 1.1]{Krug--reconstruction}. In this case, the Hitchin enhanced Fourier--Mukai transform is compatible with the usual Fourier--Mukai transform $\Phi_\cP\colon \Coh(X)\to \Coh(Y)$ in the sense that we have
\[
\For_Y\circ \Phi^Z_{\cP}\cong \Phi_\cP\circ \For_X \,,
\]
where $\For_Y\colon \Hi^{q_*p^*V}(Y)\to \Coh(Y)$ and $\For_X\colon \Hi^V(X)\to \Coh(X)$ are the forgetful functors. 
\end{remark}

\begin{example}\label{ex:FMpull}
Let $f\colon Y\to X$ be a morphism, and let $\Gamma_f=(f\times \id_Y)(Y)\subset X\times Y$ be its graph. Then, for every $L\in \Coh(Y)$ and $(E,\theta)\in \Hi^V(X)$, we have a natural isomorphism $\Phi^{\Gamma_f}_{(f\times \id_Y)_*L}\cong f^*(E,\theta)\otimes L$. In particular, $\Phi^{\Gamma_f}_{\reg_{\Gamma_f}}(E,\theta)\cong f^*(E,\theta)$.
\end{example}

\begin{lemma}\label{lem:ses}
Let $0\to \cP'\to \cP\to \cP''\to 0$ be a short exact sequence of coherent sheaves on $Z$. Then, for every Hitchin bundle $(E,\theta)\in \Hi^V(X)$, there is the following short exact sequence in $\Hi^{q'_*p'^*V}(Y)$:
\[
0\to\Phi^Z_{\cP'}(E,\theta)\to \Phi^Z_{\cP}(E,\theta)\to \Phi^Z_{\cP''}(E,\theta)\,. 
\]
\end{lemma}

\begin{proof}
 This follows from the fact that tensor products by vector bundles are exact, together with the fact that a short exact sequence of Hitchin pairs is given by a short exact sequence of the underlying coherent sheaves.
\end{proof}

Let now $i\colon Z'\hookrightarrow Z$ be a closed embedding, and let 
\[p'=p\circ i\colon Z'\to X\quad,\quad q'=q\circ i\colon Z'\to Y\]
be the restrictions of the projections to $Z'$. The canonical surjection $\reg_{Z}\to \reg_{Z'}$ induces a surjection $\phi\colon q_*p^*V\to q_*'p'^*V$. Recall that this yields a functor 
$\phi^\#\colon \Hi^{q_*'p'^*V}(X)\to \Hi^{q^*p_*V}(X)$; see \autoref{subsect:change}. 

\begin{lemma}\label{lem:FMinduced}
For $\cQ\in \Coh(Z')$ there is an isomorphism of functors
$\Phi_{i_*\cQ}^Z\cong \phi^\#\circ \Phi_{\cQ}^{Z'}$.
\end{lemma}

\begin{proof}
This follows by the projection formula for the embedding $i\colon Z'\hookrightarrow Z$, together with \autoref{lem:induced}. 
\end{proof}

Let now $f\colon T\to Y$ be a morphism of schemes, and let $Z_T:=T\times_Y Z\subset T\times X$ be the fiber product such that we have a Cartesian diagram 
\begin{equation}\label{diag:cart2}
 \xymatrix{
 Z_T\ar[r]^{f'} \ar[d] & Z \ar[d] \\
 T\ar[r]^{f} & Y\,. 
 }
\end{equation}

\begin{lemma}\label{FMbasechange}
There is an isomorphism of functors $f^*\circ \Phi^Z_{\cP}\cong \Phi^{Z_T}_{f'^*\cP}$. 
\end{lemma}

\begin{proof}
This follows from \autoref{lem:basechange}.
\end{proof}

\subsection{Stability of Hitchin pairs}

For this subsection, let $X$ be a smooth projective variety of dimension $n$. We fix an ample numerical class $H\in \NN^1(X)$. 
For every sheaf $A\in \Coh(X)$, we define its \emph{$H$-degree} and its \emph{$H$-slope} by 
\[
 \deg_H(A)=\int_X \cc_1(A)\cdot H^{n-1}\quad,\quad \mu_H(A)=\frac{\deg_H(A)}{\rank(A)}\,.
\]
A Hitchin bundle $(E,\theta)\in \Hi^V(X)$ is called \emph{($H$-slope) stable} if for every Hitchin subsheaf $A\subset E$ with $\rank(A)<\rank(E)$, we have $\mu(A)<\mu(E)$. It is called \emph{($H$-slope) semi-stable} 
if for every Hitchin subsheaf $A\subset E$, we have $\mu(A)\le \mu(E)$.

Clearly, if $E$ is (semi-)stable as an ordinary sheaf, the Hitchin pair $(E,\theta)$ is (semi-)stable too. However, the converse is not true. 

\begin{example}\label{ex:stab}
Let $X$ be a smooth projective curve, and let $L\in \Pic X$ be a line bundle of degree $-1$. We set $E:=\reg_X\oplus L$. Then $E$ is not stable and has exactly one destabilizing subsheaf, namely the direct summand $\reg_X$. Now, let 
\[
 \theta\colon E\otimes L\cong L\oplus L^{\otimes 2}\xrightarrow{\begin{pmatrix}
 0& 0\\\id_L&0
 \end{pmatrix}} \reg_X\oplus L\cong E\,,
\]
and consider the $L$-cotwisted Hitchin pair $(E,\theta)$. Then, the subbundle $\reg_X\subset E$ is not a Hitchin subbundle, hence $(E,\theta)$ is stable. 
\end{example}

\begin{remark}\label{rem:stabletwisted}
Let $(E,\theta)\in \Hi^V(X)$ be stable (or semi-stable), and let $L\in \Pic X$. Then $(E,\theta)\otimes L$ is again stable (or semi-stable). 
\end{remark}

\begin{lemma}\label{lem:inducedstab}
Let $\phi\colon V\to Q$ be a surjective morphism of vector bundles on $X$. A Hitchin pair $(E,\theta)\in \Hi^Q(X)$ is (semi-)stable if and only if $\phi^\#(E,\theta)\in \Hi^V(X)$ is (semi-)stable. 
\end{lemma}

\begin{proof}
 By \autoref{lem:induced}, $A\subset E$ is a Hitchin subsheaf of $(E,\theta)$ if and only if it is a Hitchin subsheaf of $\phi^\#(E,\theta)$. The assertion now follows from the fact that the slope of a Hitchin pair is defined as the slope of the underlying sheaf.
\end{proof}

\begin{prop}\label{prop:HNfilt}
For every $(E,\theta)\in \Hi^V(X)$, there is a unique filtration, called the \emph{Harder--Narasimhan filtration} of $(E,\theta)$, by Hitchin subsheaves
\[
E=\FF^0E\supset \FF^1 E\supset \dots\supset \FF^{m-1}E\supset \FF^{m}E=0 
\]
such that all the successive
quotients $(\FF^i E/\FF^{i+1} E,\bar\theta)$ are semi-stable Hitchin pairs with 
$\mu(\FF^i E/\FF^{i+1} E)<\mu(\FF^{i+1} E/\FF^{i+2} E)$ for every $i=0,\dots, m-2$.
\end{prop}

\begin{proof}
 See \cite[Sect.\ 3]{Simpson--moduliI}.
\end{proof}

\begin{remark}
 The Harder--Narasimhan filtration of the Hitchin pair $(E,\theta)$ does not need to agree with the Harder--Narasimhan filtration of $E$; compare \autoref{ex:stab}.
\end{remark}

\subsection{Equivariant Hitchin pairs}\label{subsect:equi}

Let $G$ be a finite group acting on a scheme $X$. A \emph{$G$-equivariant vector bundle} $(V,\alpha)$ on $X$ consists of a vector bundle $V \in \VB(X)$ and a family of isomorphisms $\alpha=\{\alpha_g\colon V\xrightarrow\sim g^*V\}$, called a \emph{$G$-linearization}, such that for all $g,h\in G$ the following diagram commutes:
\[
\xymatrix{
 V \ar^{\alpha_g}[r] \ar@/_6mm/^{\alpha_{hg}}[rrr] & g^*V \ar^{g^*\alpha_h}[r] & g^*h^*V\ar^{\cong}[r] & (hg)^*V\,.
} 
\]

\begin{remark}\label{rem:Gequi}
Let $(V,\alpha)$ be a $G$-equivariant vector bundle. For $g\in G$, we consider the composition 
\[
\Hi^V(X)\xrightarrow{g^*} \Hi^{g^*V}(X)\xrightarrow{\alpha_g^\#} \Hi^V(X)\,, 
\]
which is an autoequivalence of $\Hi^V(X)$. In the following, we simply denote this autoequivalence by $g^*\colon \Hi^V(X)\xrightarrow\sim \Hi^V(X)$. For $g,h\in G$, there is a canonical ismorphism of autoequivalences $g^*h^*\cong (hg)^*$, which means that we have an action of the group $G$ on the category $\Hi^V(X)$.   
\end{remark}

Given a $G$-equivariant vector bundle $(V,\alpha)$, a \emph{$G$-equivariant $V$-cotwisted Hitchin pair} $\bigl((E,\theta),\lambda\bigr)$ consists of a Hitchin pair $(E,\theta)\in \Hi^V(X)$ and a family of isomorphisms of Hitchin pairs $\lambda=\{\lambda_g\colon (E,\theta)\xrightarrow\sim g^*(E,\theta)\}$ such that for all $g,h\in G$ the following diagram commutes:
\[
\xymatrix{
 E \ar^{\lambda_g}[r] \ar@/_6mm/^{\lambda_{hg}}[rrr] & g^*E \ar^{g^*\lambda_h}[r] & g^*h^*E\ar^{\cong}[r] & (hg)^*E\,.
} 
\] 
That $\lambda_g$ is a morphism of Hitchin pairs means explicitly that the following diagram commutes:
\[
 \xymatrix{
E\otimes V \ar[rr]^{\lambda_g\otimes \alpha_g} \ar[d]_{\theta} && g^*E\otimes g^*V \ar[d]_{g^*\theta} \\
E\ar[rr]^{\lambda_g}  && g^*E\,. 
 }
\]
A \emph{$G$-equivariant Hitchin subsheaf} of $\bigl((E,\theta),\lambda\bigr)$ is a Hitchin subsheaf $A\subset E$ such that for every $g\in G$, we have the equality $\lambda_g(A)=g^*A$ of subsheaves of $g^*E$. 

Let $\pi\colon X\to X/G$ be a geometric quotient. For every $g\in G$, we have the equality $\pi\circ g=\pi$, which gives an isomorphism of functors $g^*\circ \pi^*\cong \pi^*$. Hence, for every $V\in \VB(X/G)$, the pull-back $\pi^*V$ carries a canonical $G$-linearisation. Furthermore, for every Hitchin pair $(F,\eta)\in \Hi^V(X/G)$, the pull-back $\pi^*(F,\eta)$ has canonically the structure of a $G$-equivariant $\pi^*V$-cotwisted Hitchin pair. 

\begin{lemma}\label{lem:stabpullbackequi}
Let a finite group $G$ act on a smooth projective variety $X$ such that $Y=X/G$ is again smooth and $\pi\colon X\to Y$ is flat. Let $H\in \NN^1(Y)$ be an ample class, $V\in \VB(Y)$, and $(F,\eta)\in \Hi^V(Y)$ a Hitchin bundle.
 \begin{enumerate}
\item\label{i} If $\mu_{\pi^*H}(A)\le \mu_{\pi^*H}(\pi_n^*F)$ holds for all $\sym_n$-equivariant Hitchin subsheaves $A$ of $\pi^*(F,\eta)$ with $\rank A<\rank F$, then $F$ is slope semi-stable with respect to $H$. 
\item\label{ii} If $\mu_{\pi^*H}(A)< \mu_{\pi^*H}(\pi_n^*F)$ holds for all $\sym_n$-equivariant Hitchin subsheaves $A$ of $\pi^*(F,\eta)$ with $\rank A<\rank F$, then $F$ is slope stable with respect to $H$. 
\end{enumerate}
\end{lemma} 
\begin{proof}
The proof is completely analogous to the proof for vector bundles without the structure of a Hitchin pair; see \cite[Lem.\ 1.1]{Krug--stab}.
\end{proof}

\section{Construction of Hitchin pairs on Hilbert schemes of points}\label{sect:Hilbert}

\subsection{Hilbert schemes of points}

Let $X$ be a smooth quasi-projective variety. For every $n\in \IN$ there is a fine moduli 
space $X^{[n]}$ of zero-dimensional sub-schemes of $X$ of length $n$, called the \emph{Hilbert 
scheme of $n$ points on $X$} (also called punctual Hilbert scheme). This means that there is a closed
sub-scheme $\Xi_n\subset 
X\times X^{[n]}$ which is flat and finite of degree $n$ over $X^{[n]}$, called the 
\emph{universal family}, such that, for every scheme $T$ and every closed sub-scheme $Z\subset 
X\times T$ which is flat and finite of degree $n$ over $T$, there is a \emph{classifying 
morphism} $f\colon T\to X^{[n]}$ such that $Z=(\id_X\times f)^{-1} \Xi_n$, where 
$(\id_X\times f)^{-1} \Xi_n\cong T\times_{X^{[n]}} \Xi_n$ is the scheme-theoretic preimage.

There is the \emph{Hilbert--Chow morphism} $\mu\colon X^{[n]}\to X^{(n)}$ to the symmetric 
product $X^{(n)}:=X^n/\sym_n$ that sends any zero-dimensional sub-scheme 
$\xi\subset X$ to its weighted support: $\mu([\xi])=\sum_{x\in \xi} \ell(\xi,x)\cdot x$.

If $X=C$ is a smooth curve, $\mu$ is an isomorphism, hence $C^{[n]}\cong C^{(n)}$. If $X$ is a smooth surface, $X^{[n]}$ is smooth and $\mu\colon X^{[n]}\to X^{(n)}$ is a crepant resolution of the singularities of the symmetric product.

\subsection{Tautological bundles and Hitchin pairs}\label{subsect:taut}

Let $X$ be a smooth quasi-projective variety, let $n\in \IN$, and let $X\xleftarrow p 
\Xi_n\xrightarrow q X^{[n]}$ be the projections from the universal family of the Hilbert 
scheme $X^{[n]}$. Given a sheaf $E\in \Coh(X)$, the associated \emph{tautological sheaf} is 
defined by $E^{[n]}=q_*p^*E\in \Coh(X^{[n]})$. Over a point $[\xi]\in X^{[n]}$ corresponding 
to a zero-dimensional sub-scheme $\xi\subset X$ of length $n$, the fiber of the tautological 
sheaf is given by $E^{[n]}([\xi])=\Ho^0(E_{\mid\xi})$. If $E$ is a vector bundle, then $E^{[n]}$ 
is again a vector bundle with $\rank E^{[n]}=n\rank E$. This follows from the fact that 
$q\colon \Xi_n\to X^{[n]}$ is flat and finite of degree $n$.

Let $V\in \VB(X)$. If $E$ carries the structure of a $V$-cotwisted Hitchin pair, we get an induced structure of a $V^{[n]}$-cotwisted Hitchin pair on $E^{[n]}$. More precisely, for 
$(E,\theta)\in \Hi^V(X)$, we define the associated \emph{tautological Hitchin pair}, using 
the Hitchin enhanced Fourier--Mukai transform as introduced in \autoref{subsect:FM}, by \[
 (E^{[n]},\theta^{\{n\}}):=(E,\theta)^{[n]}:=\Phi_{\reg_{\Xi_n}}^{\Xi_n}(E,\theta)\in \Hi^{V^{[n]}}(X^{[n]})\,.
\]

\begin{remark}\label{rem:taut}
To provide some intuition, let us give a concrete description of the induced map $\theta^{\{n\}}=q_*p^*\theta\circ \nu_{q,E,V}\colon E^{[n]}\otimes V^{[n]}\to E^{[n]}$. For a zero-dimensional sub-scheme $\xi\subset X$ of length $n$, the fiber of $\theta^{\{n\}}$ over the point $[\xi]\in X^{[n]}$ is given by the composition 
\[
\theta^{\{n\}}([\xi])\colon (E^{[n]}\otimes V^{[n]})([\xi])\cong \Ho^0(E_{\mid \xi})\otimes \Ho^0(E_{\mid \xi})\xrightarrow\cup \Ho^0((E\otimes V)_{\mid \xi})\xrightarrow{\theta_{\mid \xi}} \Ho^0(E_{\mid \xi})\cong E^{[n]}([\xi])\,. 
\]
The fact that $\bigl(\theta^{\{n\}}\circ (\theta^{\{n\}}\otimes \id_{V^{[n]}})
\bigr)_{\mid E^{[n]}\otimes \bigwedge^2 V^{[n]}}=0$, which already follows
from \autoref{lem:push}, can also be read off from the above description of $\theta^{\{n\}}$.
\end{remark}
 
\begin{lemma}\label{lem:baseuniversal}
 Let $T$ be a scheme, let $Z\subset X\times T$ be a sub-scheme which is flat and finite of degree $n$ over $T$, and let $f\colon T\to X^{[n]}$ be the classifying morphism for $Z$. Then, for $(E,\theta)\in \Hi^V(X)$, we have a natural isomorphism of $f^*V^{[n]}$-cotwisted Hitchin pairs
 \[
 f^*(E,\theta)^{[n]}\cong \Phi^Z_{\reg_Z}(E,\theta)
 \]
\end{lemma}

\begin{proof}
 This follows by applying \autoref{FMbasechange} to the Cartesian diagram
 \begin{equation*}
 \xymatrix{
 Z\ar[r]^{} \ar[d]_{} & \Xi_n \ar[d]^{} \\
 T\ar[r]^{f} & X^{[n]}\,. 
 }
\end{equation*} 
\end{proof}

\subsection{Tangent bundle of the Hilbert scheme}

If $X$ is a smooth curve or surface, the Hilbert scheme $X^{[n]}$ is again smooth.
In this case, one might hope that, if we start with a sheaf $E\in \Coh(X)$ and a non-zero Higgs field $E\to E\otimes \Omega_X$, we get an induced non-zero Higgs field $E^{[n]}\to E^{[n]}\otimes \Omega_{X^{[n]}}$ on $E^{[n]}$. However, as the following example shows, there cannot be such a general construction. 

\begin{example}\label{ex:noHiggs}
 Let $X=\IP^1$. Then, as $\Omega_X=\reg_{\IP^1}(-2)$, the bundle $E:=\reg_{\IP^1}(-1)\oplus\reg_{\IP^1}(1)$ has a non-zero Higgs field. There is an isomorphism $(\IP^1)^{[2]}\cong (\IP^1)^{(2)}\cong \IP^2$ under which the tautological bundles are given by 
\[
\reg_{\IP^1}(-1)^{[2]}\cong \reg_{\IP^2}(-1)^{\oplus 2}\quad,\quad \reg_{\IP^1}(1)^{[2]}\cong \reg_{\IP^2}^{\oplus 2}\,; 
\]
see \cite[Sect.\ 3]{Nagaraj--sym}. Hence, $E^{[2]}\cong \reg_{\IP^2}(-1)^{\oplus 2}\oplus \reg_{\IP^2}^{\oplus 2}$. Using the Euler sequence, one computes $\Ho^0(\Omega_{\IP^2}(-1))=\Ho^0(\Omega_{\IP^2})=\Ho^0(\Omega_{\IP^2}(1))=0$. It follows that 
$\Hom(E^{[2]},E^{[2]}\otimes \Omega_{\IP^2})=0$, which means that there is no non-zero Higgs field on $E^{[2]}$, though there is one on $E$.
\end{example}

However, our construction from the previous subsection equips $E^{[n]}$ with the structure of a logarithmic Higgs sheaf, with respect to the boundary divisor $B_n\subset X^{[n]}$, whenever there is a Higgs field on $E$. Indeed, let $(E,\theta)$ be a Higgs sheaf, i.e.,\ a $\TT_X$-cotwisted Hitchin pair. Then $(E,\theta)^{[n]}$ is a $(\TT_X)^{[n]}$-cotwisted Hitchin pair. Hence, our assertion follows from the following result. 

\begin{theorem}[\cite{Stapletontaut}]
Let $X$ be a smooth curve or surface. Then there is an isomorphism 
\[
 (\TT_X)^{[n]}\cong \TT_{X^{[n]}}(-\log B_n)\,.
\]
\end{theorem}

\begin{proof}
For $X$ a smooth surface, this is \cite[Thm.\ B]{Stapletontaut}. The proof in the case case where $X$ is a 
smooth curve is essentially the same. The only difference is that we do not need to restrict to an open subset 
$U\subset X^{[n]}$ (as done in the surface case on top of page 1181 of \cite{Stapletontaut}.) since, in the curve 
case, the universal family $\Xi_n$ is already smooth everywhere (for $X$ a smooth curve, we have $\Xi_n\cong 
X\times X^{(n-1)}$). 
\end{proof}

\section{Reconstruction in the case of curves}\label{sect:reccurves}

\subsection{Curves of genus $g\ge 2$}\label{subsect:g2}

In this subsection, we will prove \autoref{thm:g2}. So let $C$ be a smooth projective curve of genus $\genus(C)\ge 2$, let $V\in \VB(C$), and let $(E,\theta)\in \Hi^V(C)$ be a semi-stable Hitchin bundle.
We enhance the reconstruction method of \cite{Biswas-Nagaraj--reconstructioncurves} and \cite[Sect.\ 2]{Biswas-Nagaraj--reconstructionsurfaces}, which reconstructs vector bundles $E$ on $C$ from their tautological images $E^{[n]}$ on the symmetric product $C^{(n)}$, to Hitchin bundles. 

Let $C\cong \Delta\subset C\times C$ be the diagonal with ideal sheaf $\cI_{\Delta}$. We denote the $(n-1)$-th infinitesimal thickening of the diagonal by $Z:=V(\cI_\Delta^n)\subset C\times C$. Via the projection $\pr_1\colon C\times C\to C$, the sub-scheme $Z\subset C\times C$ is a flat family of length $n$ sub-schemes of $C$ over $C$. Hence, we get a classifying morphism $\iota\colon C\to C^{(n)}$, which is a closed embedding with image the small diagonal in $C^{(n)}$. By \autoref{lem:baseuniversal}, we have 
\[
 \iota^*(E,\theta)^{[n]}\cong \Phi^Z_{\reg_Z}(E,\theta)\,.
\]
Hence, it is sufficient to show that the isomorphism class of $(E,\theta)\in \Hi^V(C)$ is determined by $\Phi^Z_{\reg_Z}(E,\theta)\in \Hi^{b_*a^*V}(C)$, where $a,b\colon Z\to C$ are the restrictions of the projections $\pr_1,\pr_2\colon C\times C\to C$. 

The structure sheaf $\reg_{Z}$ has the filtration
\[
 \reg_{Z}=\cJ_\Delta^0\supset \cJ_\Delta^1\supset \dots \supset\cJ_\Delta^{n-1}\supset \cJ_\Delta^{n}=0\,,
\]
where $\cJ_\Delta:=\cI_\Delta/\cI_\Delta^{n}$. This induces a filtration by Hitchin subsheaves
\begin{equation}\label{eq:HNJet}
\Phi^Z_{\reg_Z}(E,\theta)= \Phi^Z_{\cJ_\Delta^0}(E,\theta)\supset
\Phi^Z_{\cJ_\Delta^1}(E,\theta)\supset \dots\supset
\Phi^Z_{\cJ_\Delta^{n-1}}(E,\theta)\supset \Phi^Z_{\cJ_\Delta^n}(E,\theta)=0
\end{equation}
which we will show to be the Harder--Narasimhan filtration of $\Phi^Z_{\reg_Z}(E,\theta)\in \Hi^{b_*a^*V}(C)$. Note that we have short exact sequences
\begin{equation}\label{eq:sesideals}
 0\to \cJ_\Delta^{k+1}\to \cJ_{\Delta}^{k}\to i_*\omega_C^{\otimes k}\to 0
\end{equation}
where $i\colon C\hookrightarrow Z$ is the closed embedding of the reduced diagonal. By \autoref{lem:FMinduced} combined with \autoref{ex:FMpull}, we have $\Phi^{Z}_{i_*\omega_C^{\otimes k}}\cong \phi^\#\bigl((E,\theta)\otimes \omega_C^{\otimes k}\bigr)$ where $\phi\colon b_*a^*V\twoheadrightarrow V$ is the canonical surjection induced by $\reg_Z\twoheadrightarrow\reg_\Delta$. Hence, applying \autoref{lem:ses} to \eqref{eq:sesideals} gives a short exact sequence
\begin{equation*}
 0\to \Phi^Z_{\cJ_\Delta^{k+1}}(E,\theta)\to \Phi^Z_{\cJ_\Delta^k}(E,\theta)\to \phi^\#\bigl((E,\theta)\otimes \omega_C^{\otimes k}\bigr)\to 0\,.
\end{equation*}
By the assumption that $(E,\theta)$ is semi-stable, \autoref{rem:stabletwisted}, and \autoref{lem:inducedstab}, we see that the factor 
$\Phi^Z_{\cJ_\Delta^{k}}(E,\theta)/ \Phi^Z_{\cJ_\Delta^{k+1}}(E,\theta)\cong \phi^\#\bigl((E,\theta)\otimes \omega_C^{\otimes k}\bigr)$ is semi-stable. Furthermore, since $\deg(\omega_C)>0$ by the assumption on the genus $\genus(C)$, we have strict inequalities \[\mu\left(\Phi^Z_{\cJ_\Delta^{k}}(E,\theta)/ \Phi^Z_{\cJ_\Delta^{k+1}}(E,\theta)\right)
< \mu\left(\Phi^Z_{\cJ_\Delta^{k+1}}(E,\theta)/ \Phi^Z_{\cJ_\Delta^{k+2}}(E,\theta)\right)\,.\] 
This means that 
\eqref{eq:HNJet} is indeed the Harder--Narasimhan filtration of $\Phi^Z_{\reg_Z}(E,\theta)\in 
\Hi^{b_*a^*V}(X)$; see \autoref{prop:HNfilt}. The uniqueness of the first factor $\phi^\#(E,\theta)$ of the 
Harder--Narasimhan filtration, together with \autoref{lem:induced}(i), now show that the isomorphism class of 
$(E,\theta)$ is determined by $\Phi^Z_{\reg_Z}(E,\theta)\cong \iota^*(E,\theta)^{[n]}$.

\begin{remark}
In analogy with \cite[Sect.\ 2]{Biswas-Nagaraj--reconstructionsurfaces}, we can relax the assumptions of \autoref{thm:g2} as follows. Let $(E,\theta)\in \Hi^V(X)$ be a Hitchin bundle with Harder--Narasimhan filtration 
\[
E=\FF^0E\supset \FF^1 E\supset \dots\supset \FF^{m-1}E\supset \FF^{m}E=0 \,.
\]
We set $\mu_{\mathsf{min}}(E,\theta):=\mu(E/\FF^{1} E)$ and $\mu_{\mathsf{max}}(E,\theta):=\mu(\FF^{m-1}E)$.
Then, instead of assuming in \autoref{thm:g2} that $(E,\theta)$ and $(F,\eta)$ are semi-stable, it is sufficient to assume that 
\[
\mu_{\mathsf{max}}(E,\theta)-\mu_{\mathsf{min}}(E,\theta)<2(\genus(C)-1)\quad\text{and}\quad \mu_{\mathsf{max}}(F,\theta)-\mu_{\mathsf{min}}(F,\theta)<2(\genus(C)-1)\,. 
\]
\end{remark}

\subsection{Elliptic curves}

In this subsection, we prove
\autoref{thm:g1} as a consequence of the following more general result; compare \cite[Thm.\ 5.2]{Krug--reconstruction}.

\begin{prop}\label{prop:autom}
Let $X$ be a smooth projective variety which has $n+1$ automorphisms $\sigma_0,\dots,\sigma_n\in \Aut(X)$ \emph{with empty pairwise equalizers}, which means that $\sigma_i(x)\neq \sigma_j(x)$ for all $i\neq j$ and all $x\in X$. Then, for every $n\in \IN$ and every two Hitchin pairs $(E,\theta),(F,\eta)\in \Hi^V(X)$, we have 
 \[
 (E,\theta)^{[n]}\cong (F,\eta)^{[n]}\quad\implies \quad (E,\theta)\cong (F,\eta)\,. 
 \]
\end{prop}

\begin{proof}
Replacing $\sigma_i$ by $\sigma_0^{-1}\circ \sigma_i$, we can assume without loss of generality that $\sigma_0=\id_X$. By the assumption on the automorphisms, the graphs $\Gamma_{\sigma_i}$ are pairwise disjoint. Hence, for every $j\in \{0,\dots, n\}$, the union
\[G_j=\bigsqcup_{\substack{i=0,\dots,n\\i\neq j}}\Gamma_{\sigma_i}\subset X\times X\]
is flat and finite of degree $n$ over $X$. Let $f_j\colon X\to X^{[n]}$ be the classifying morphism for $G_j$. 
By \autoref{lem:baseuniversal}, we have $f_j^*(E,\theta)^{[n]}\cong \Phi^{G_j}_{\reg_{G_j}}(E,\theta)$. Since the $\Gamma_{\sigma_j}$ are disjoint, we have $\reg_{G_j}\cong \oplus_{i\neq j}\reg_{\Gamma_{\sigma_i}}$. 
Together with \autoref{lem:FMinduced} and \autoref{ex:FMpull}, this gives
\begin{equation}\label{eq:fj}
f_j^*(E,\theta)^{[n]}\cong \Phi^{G_j}_{\reg_{G_j}}(E,\theta)\cong \bigoplus_{i\neq j} \phi^{j\#}_i\Phi^{\Gamma_{\sigma_i}}_{\reg_{\Gamma_{\sigma_i}}}(E,\theta)\cong \bigoplus_{i\neq j} \phi^{j\#}_i\sigma_i^*(E,\theta),
\end{equation}
where $\phi^j_i\colon \Phi_{\reg_{G_j}}(V)\cong \oplus_{i\neq j}\sigma_i^*V\to \Phi_{\reg_{\Gamma_{\sigma_i}}}(V)\cong \sigma_i^*V$ is the projection, which is induced by the surjection $\reg_{G_j}\to \reg_{\Gamma_{\sigma_i}}$. We also set $\mathbb V:=\oplus_{i=1}^n\sigma_i^*V$, and write 
\[
 \psi_j\colon \mathbb V\to \bigoplus_{\substack{i=0,\dots,n\\i\neq j}}\sigma_i^*V\quad,\quad \phi_k=\phi^j_k\circ \psi_k\colon \mathbb V \to \sigma_k^*V 
\]
for the appropriate projections to the summands. Applying $\psi_j^\#\colon \Hi^{\oplus_{i\neq j} \sigma_i^* V}(X)\to \Hi^{\mathbb V}(X)$ to \eqref{eq:fj} gives an isomorphism 
\[
\psi_j^\#f_j^*(E,\theta)^{[n]}\cong \bigoplus_{i\neq j} \phi^{\#}_i\sigma_i^*(E,\theta)\cong \bigoplus_{i\neq j} \sigma_i^*\phi^{\#}_0(E,\theta)\,. 
\]
Here, the $\sigma_i^*$ on the right-hand side is a shortcut for the autoequivalence $\alpha_{\sigma_i}^\#\circ \sigma_i^*$ of $\Hi^{\mathbb V}(X)$ where $\alpha_{\sigma_i}\colon \mathbb V\to \sigma_i^* \mathbb V$ is the canonical isomorphism given by permutation of the summands; compare \autoref{rem:Gequi}. 

Considering these isomorphisms for varying $j=0,\dots, n$ gives
\begin{equation}\label{eq:implication1}
(E,\theta)^{[n]}\cong (F,\eta)^{[n]}\quad\implies\quad \bigoplus_{\substack{i=0,\dots,n\\i\neq j}}\sigma_i^*\phi^{\#}_0(E,\theta)\cong \bigoplus_{\substack{i=0,\dots,n\\i\neq j}}\sigma_i^*\phi^{\#}_0(F,\eta)\quad \forall\, j=0,\ldots, n\,.
\end{equation}
Since the category $\Hi^{\mathbb V}(X)$ is Krull-Schmidt, \cite[Prop.\ 5.4]{Krug--reconstruction} gives
\begin{equation}\label{eq:implication2}
 \bigoplus_{\substack{i=0,\dots,n\\i\neq j}}\sigma_i^*\phi^{\#}_0(E,\theta)\cong \bigoplus_{\substack{i=0,\dots,n\\i\neq j}}\sigma_i^*\phi^{\#}_0(F,\eta)\quad \forall j=0,\ldots, n\,\,\implies\quad \phi^{\#}_0(E,\theta)\cong \phi^{\#}_0(F,\eta)\,.
\end{equation}
Combining implication \eqref{eq:implication1} with implication \eqref{eq:implication2} and \autoref{lem:induced}(i) proves the assertion. 
\end{proof}

\autoref{thm:g1} is a special case of \autoref{prop:autom} since, on an elliptic surface, there is an infinite group of automorphisms with empty pairwise equalizers, namely the transpositions.

\section{Stability in the case of curves}\label{sect:stabcurves}

In this section, we prove \autoref{thm:curvestab}. The special case of bundles without the structure of a 
Hitchin pair was proved in \cite{Krug--stab}, and here we explain the extra ingredients necessary for the same 
proof to work for Hitchin pairs.

Let $C$ be a smooth projective curve, and let $(E,\theta)\in \Hi^V(C)$ be a stable (or semi-stable) Hitchin 
pair with $\mu(E)\notin [-1,n-1]$ (or $\mu(E)\notin (-1,n-1)$). Let $\pi_n\colon C^n\to C^{(n)}$ be the 
$\sym_n$-quotient morphism. As discussed in \autoref{subsect:equi}, the pull-back $\pi_n^*(E,\theta)^{[n]}\in 
\Hi^{\pi_n^*V^{[n]}}(C^n)$ of the tautological Hitchin bundle carries a canonical $\sym_n$-linearization. By 
\autoref{lem:stabpullbackequi}, for $(E,\,\theta)^{[n]}$ to be stable (or semi-stable), it is sufficient to 
prove that, for every $\sym_n$-equivariant Hitchin subsheaf $A\subset \pi_n^*E^{[n]}$ with $\rank A< 
\rank(\pi_n^*E^{[n]})\,=\,n\rank(E)$, we have the inequality $\mu_{\wH_n}(A)\,<\, \mu_{\wH_n}(\pi^* E^{[n]})$ (or 
$\mu(A)\,\le\, \mu(\pi^* E^{[n]})$).

The key to the proof of the stability criterion for tautological bundles in \cite{Krug--stab} is the 
existence of a short exact sequence
\begin{align}\label{eq:keyses}
0\to \pr_1^* E(-\delta_n(1))\xrightarrow{i_E} \pi_n^* E^{[n]}\xrightarrow{p_E} \overline\pr_{1}^*\pi_{n-1}^*E^{[n-1]}\to 0\,, 
\end{align}
where $\pr_1\colon C^n\to C$ and $\overline\pr_1\colon C^n\to C^{n-1}$ are the projections to the first factor and to the last $n-1$ factors, respectively, and $\delta_n(1)=\sum_{i=2}^n\Delta_{1i}$ where 
\[
 \Delta_{1i}=\{(x_1,\dots,x_n)\in X^n\mid x_1=x_i\}\,;
\]
see \cite[Prop 1.5]{Krug--stab}.
Let now $A\subset \pi^*_nE^{[n]}$ be an equivariant Hitchin subsheaf. Setting $A':=i_E^{-1}A\subset \pr_1^* E(-\delta_n(1))$ and $A'':=p_E(A)$, we get a commutative diagram
\begin{align}\label{diag:full}
\xymatrix{
 & 0 \ar[d] & 0 \ar[d] & 0 \ar[d] &\\
0\ar[r] & A' \ar[r] \ar[d] & A\ar[r]\ar[d] & A''\ar[r]\ar[d]& 0\\
0\ar[r] & \pr_1^* E(-\delta_n(1))\ar^{\quad i_E}[r] & \pi_n^* E^{[n]}\ar^{p_E\quad}[r]& \overline \pr_{1}^*\pi_{n-1}^*E^{[n-1]}\ar[r] & 0}\,.
\end{align}
with exact columns and rows.
The proof in \cite{Krug--stab} uses the (semi-)stability of $E$ (by assumption) and the (semi-)stability of $E^{[n-1]}$ (by induction) to find upper bounds for the slopes of $A'$ and $A''$, which lead to the desired inequality $\mu_{\wH_n}(A)< \mu_{\wH_n}(\pi^* E^{[n]})$ (or $\mu(A)\le \mu(\pi^* E^{[n]})$). However, we only know the semi-stability of $(E,\theta)$ as a Hitchin pair, which does not imply the stability of $E$ as an ordinary vector bundle; see \autoref{ex:stab}. The reason that, nevertheless, the proof of \emph{loc.\ cit.}\ works for Hitchin pairs is that we can enhance \eqref{eq:keyses} to a short exact sequence of $\pi_n^*V^{[n]}$-cotwisted Hitchin pairs. 

For this, we need to have a look at the construction of the short exact sequence \eqref{eq:keyses} in the proof 
of \cite[Prop 1.5]{Krug--stab}. We have $(\pi_n\times \id_{C})^{-1}\Xi_n=D_n$, where $D_n=\bigcup_{k=1}^n 
\Gamma_{k}\subset C^n\times C$ is the union of the graphs $\Gamma_k:=\Gamma_{\pr_k}$ of the projections 
$\pr_k\colon C^n\to C$ to the $k$-th factor. It follows by flat base change that $E^{[n]}\cong 
\Phi_{\reg_{D_n}}(E)$. Under this isomorphism, \eqref{eq:keyses} is induced by the canonical short exact sequence 
of Fourier--Mukai kernels
\begin{align}\label{eq:Dses}
 0\to \reg_{\Gamma_{\pr_1}}(-\sum_{k=2}^n [\Gamma_{\pr_k}\cap \Gamma_{\pr_1}])\to \reg_{D_n}\to \reg_{\cup_{k=2}^n\pr_k}\to 0
\end{align} 
together with the isomorphism $\reg_{\cup_{k=2}^n\pr_k}\cong (\overline\pr_1\times\id_C)^*\reg_{D_{n-1}}$.
In other words, the short exact sequence \eqref{eq:keyses} is isomorphic to the exact sequence 
\begin{align}\label{eq:keyses2}
0\to \Phi_{\reg_{\Gamma_{\pr_1}}(-\sum_{k=2}^n [\Gamma_{\pr_k}\cap \Gamma_{\pr_1}])}(E)\to \Phi_{\reg_{D_n}}(E)\to \Phi_{\reg_{\cup_{k=2}^n\pr_k}}(E)\to 0\,. 
\end{align}
By \autoref{lem:ses} we get an enhanced version of \eqref{eq:keyses2} in the form of a short exact sequence
\begin{align}\label{eq:keyses3}
0\to \Phi^{D_n}_{\reg_{\Gamma_{\pr_1}}(-\sum_{k=2}^n [\Gamma_{\pr_k}\cap \Gamma_{\pr_1}])}(E,\theta)\to \Phi^{D_n}_{\reg_{D_n}}(E,\theta)\to \Phi^{D_n}_{\reg_{\cup_{k=2}^n\pr_k}}(E,\theta)\to 0 
\end{align}
of $\pi_n^*V^{[n]}$-cotwisted Hitchin pairs. The surjections $\reg_{D_n}\to\reg_{\Gamma_1}$ and $\reg_{D_n}\to \reg_{\cup_{k=2}^n\pr_k}\cong (\overline\pr_i\times\id_C)^*\reg_{D_{n-1}}$ induce surjections
\[
 \phi\colon \pi_n^*V^{[n]}\to \pr_1^*V\quad,\quad\psi\colon \pi_n^*V^{[n]}\to \overline\pr_1^*\pi_{n-1}^*V^{[n-1]}\,.
\]
Applying \autoref{ex:FMpull} together with \autoref{lem:FMinduced} to the first term of \eqref{eq:keyses3}, \autoref{FMbasechange} (recalling the isomorphism $(\pi_n\times \id_{C})^*\reg_{\Xi_n}\cong\reg_{D_n}$) to the second term of \eqref{eq:keyses3}, and \autoref{FMbasechange} (recalling the isomorphism $\reg_{\cup_{k=2}^n\pr_k}\cong (\overline\pr_i\times\id_C)^*\reg_{D_{n-1}}$) together with \autoref{lem:FMinduced} to the third term of \eqref{eq:keyses3}, we see that \eqref{eq:keyses3} is isomorphic to 
\begin{align}\label{eq:Hitchinkeyses}
0\to \phi^\#\pr_1^* (E,\theta)\otimes \reg(-\delta_n(1))\xrightarrow{i_E} \pi_n^* (E,\theta)^{[n]}\xrightarrow{p_E} \psi^\#\overline\pr_{1}^*\pi_{n-1}^*(E,\theta)^{[n-1]}\to 0\,.
\end{align}

\begin{lemma}\label{lem:A'''}
Let, as above, $A$ be an $\sym_n$-equivariant Hitchin subsheaf of $\pi^*_n(E,\theta)^{[n]}$, and let $A':=i_E^{-1}(A)\subset \pr_1^* E(-\delta_n(1))$ and $A'':=p_E(A)\subset \overline\pr_1^*\pi_{n-1}^*E^{[n-1]}$.
\begin{enumerate}
 \item $A'$ is a Hitchin subsheaf of the $\pr_1^*V$-cotwisted Hitchin pair $\pr_1^* (E,\theta)\otimes \reg(-\delta_n(1))$.
 \item $A''$ is a Hitchin subsheaf of the $\overline\pr_{1}^*\pi_{n-1}^*V^{[n-1]}$-cotwisted Hitchin pair $\overline\pr_{1}^*\pi_{n-1}^*(E,\theta)^{[n-1]}$. 
\end{enumerate}
\end{lemma}

\begin{proof}
 Since we now have the short exact sequence of Hitchin pairs \eqref{eq:Hitchinkeyses}, this follows directly from \autoref{lem:subsheaves} and \autoref{lem:induced}(iii).
\end{proof}

Now, as in \cite[Subsect.\ 2.2]{Krug--stab}, we choose the points $x_2,\dots, x_n$ in such a way that all rows and columns of the diagram \eqref{diag:full} stay exact after pull-back along the embedding
\[
\iota\colon C\hookrightarrow C^n\quad,\quad\iota(t)=(t,x_2,\dots, x_n)\,. 
\]
Since $\iota^*\reg_{C^n}(-\delta_n(1))\cong\reg_C(-\sum_{i=2}^nx_i)$ and $\pr_1\circ \iota= \id_C$, we have an isomorphism of $\iota^*\pr_1^*V\cong V$-cotwisted Hitchin pairs 
\[
 \iota^*\pr_1^* \bigl((E,\theta)\otimes \reg(-\delta_n(1))\bigr)\cong (E,\theta)\otimes \reg(-\sum_{i=2}^nx_i) 
\]
By \autoref{lem:A'''}(i), $\iota^*A'$ is a Hitchin subsheaf of $(E,\theta)\otimes \reg(-\sum_{i=2}^nx_i)$.
Since, by assumption, $(E,\theta)$ is stable (or semi-stable), so is $(E,\theta)\otimes \reg(-\sum_{i=2}^nx_i)$; see \autoref{rem:stabletwisted}. Hence, 
 we have $\mu(\iota^*A')<\mu(E(-\sum_{i=2}^nx_i))$ (or $\mu(\iota^*A')\le\mu(E(-\sum_{i=2}^nx_i))$). Now, in the case $\mu(E)>n-1$ (or $\mu(E)\ge n-1$), we can follow the computations of \cite[Subsect.\ 2.2]{Krug--stab} to see that $\mu_{\wH_n}(A)<\mu_{\wH_n}(\pi_n^*E^{[n]})$ (or $\mu_{\wH_n}(A)\le \mu_{\wH_n}(\pi_n^*E^{[n]})$) which, as discussed above, by \autoref{lem:stabpullbackequi} implies that $E^{[n]}$ is stable (or semi-stable).

If $\mu(E)\le -1$, as in \cite[Subsect.\ 2.3]{Krug--stab}, we consider the closed embedding 
\[
 \iota\colon C^{n-1}\hookrightarrow C^n\quad,\quad \iota(t_2,\dots, t_n)=(x,t_2,\dots, t_n)\,,
\]
where, again, $x\in C$ is chosen in such a way that all rows and columns of the diagram \eqref{diag:full} stay exact after pull-back along $\iota^*$. Since $\overline\pr_1\circ \iota=\id_{C^{n-1}}$, we have an isomorphism of $\iota^*\overline\pr_{1}^*\pi_{n-1}^*V^{[n-1]}\cong \pi_{n-1}^*V^{[n-1]}$-cotwisted Hitchin pairs
\[
\iota^*\overline\pr_{1}^*\pi_{n-1}^*(E,\theta)^{[n-1]}\cong \pi_{n-1}^*(E,\theta)^{[n-1]}\,. 
\]
By induction, we may assume that $\pi_{n-1}^*(E,\theta)^{[n-1]}$ is stable (or semi-stable). By 
\autoref{lem:A'''}(ii), $\iota^*A''$ is a Hitchin subsheaf of $\pi_{n-1}^*(E,\theta)^{[n-1]}$.
Hence, we get that
$\mu(\iota^*A'')<\mu(\pi_{n-1}^*E^{[n-1]})$ (or $\mu(\iota^*A'')\le\mu(\pi_{n-1}^*E^{[n-1]})$). Now, the assertion in the case $\mu(E)<-1$ (or $\mu(E)\le -1$) follows from the computations in \cite[Subsect.\ 2.3]{Krug--stab}. 

\section{Higher dimensions}

\subsection{Reconstruction for higher dimensions}

In this section, we proof \autoref{thm:rechigher}.
A version of this statement for coherent sheaves without the structure of a Hitchin bundle was proved in \cite[Thm.\ 3.2 \& Rem.\ 3.5]{Krug--reconstruction}. Since the arguments which allow to lift the results to Hitchin pairs are very similar to those of the previous two sections, we will restrict ourselves to a quiet terse explanation this time. 

Recall that there is the Hilbert--Chow morphism $\mu\colon X^{[n]}\to X^{(n)}$ and the quotient morphism $\pi\colon X^n\to X^{(n)}$. We consider the open subsets 
\[
 X^{n}_0:=\{(x_1,\dots,x_n)\in X^n\mid x_i\neq x_j \text{ for $i\neq j$}\}\subset X^n\,,
\]
$X^{(n)}_0:=\pi(X^n_0)\subset X^{(n)}$, and $X^{[n]}_0:=\mu^{-1}(X^{(n)}_0)\subset X^{[n]}$. This gives the following diagram with Cartesian squares 
\[
\xymatrix{
 X^{[n]}_0 \ar^{\mu_0}[r] \ar^{i}[d] & X^{(n)}_0\ar[d] & X^n_0 \ar_{\pi_0}[l] \ar^{j}[d]\\
 X^{[n]}\ar^{\mu}[r] & X^{(n)}& X^n \ar_{\pi}[l]
}
\]
where all the vertical arrows are open immersions, and $\mu_0$ is an isomorphism.

For $(E,\theta)\in \Hi^V(X)$, we define the $\sym_n$-equivariant $(\oplus_{i=1}^n\pr_i^*V)$-cotwisted Hitchin pair $\CC(E,\theta)=(\oplus_{i=1}^n\pr_i^*E,\tilde\theta, \lambda)$ by 
\[
 \tilde\theta\colon \left(\bigoplus_{i=1}^n\pr_i^* E\right)\otimes \left(\bigoplus_{i=1}^n\pr_i^* V\right)\cong \bigoplus_{i,j=1}^n\pr_i^*E\otimes \pr_j^*V\to \bigoplus_{i=1}^n\pr_i^*(E\otimes V)\xrightarrow{\oplus \pr_i^*\theta_i} \bigoplus_{i=1}^n\pr_i^*E\,,
\]
where the middle map is the projection to the appropriate factors, and 
$\lambda_g\colon \oplus_i \pr_i^*E\to g^*(\oplus_i\pr_i^*E)$ is the direct sum of the canonical isomorphisms $\pr_i^*E\xrightarrow\sim g^*\pr_{g(i)}^*E$.

Note that the open subset $X^n_0\subset X^n$ is stable under the $\sym_n$-action on $X^n$, hence the morphism $\pi_0\colon X^n_0\to X^{(n)}_0$ is a $\sym_n$-quotient. Thus, as explained in \autoref{subsect:equi}, the Higgs bundle $j_*\pi_0^*\mu_{0*}i^*(E,\theta)^{[n]}$ carries a canonical $\sym_n$-linearization.

\begin{lemma}\label{lem:C}
For every $(E,\theta)\in \Hi^V(X)$ such that $E$ is reflexive, we have an isomorphism of $\sym_n$-equivariant $(\oplus_{i=1}^n\pr_i^*V)$-cotwisted Hitchin pairs 
\[j_*\pi_0^*\mu_{0*}i^*(E,\theta)^{[n]}\cong \CC(E,\theta)\,.\]
\end{lemma}

\begin{proof}
For vector bundles without the structure of a Hitchin pair, this is \cite[Lem.\ 1.1]{Stapletontaut}. For the straightforward generalization to reflexive sheaves, see \cite[Sect.\ 3.1]{Krug--reconstruction}. That the two structures of Hitchin pairs $\left(\bigoplus_{i=1}^n\pr_i^* E\right)\otimes \left(\bigoplus_{i=1}^n\pr_i^* V\right)\to \bigoplus_{i=1}^n\pr_i^*E$ on both sides of the alleged isomorphism agree can be checked on the open dense subset $X^n_0\subset X^n$. This can be done either using the results on Hitchin enhanced Fourier--Mukai transforms of \autoref{subsect:FM}, or the concrete description of $\theta^{\{n\}}$ of \autoref{rem:taut}. 
\end{proof}

\begin{proof}[Proof of \autoref{thm:rechigher}] 
Let $\delta\colon X\hookrightarrow X^n$, $\delta(x)=(x,\dots,x)$ be the embedding of the small diagonal. 
We have $\delta^*\CC(E,\theta)\cong (E,\theta)^{\oplus n}$ with the induced $\sym_n$-linearization given by permutation of the direct summands. Hence, the $\sym_n$-invariants are given by $\bigl(\delta^*\CC(E,\theta)\bigr)^{\sym_n}\cong (E,\theta)$. By \autoref{lem:C}, this implies
\[
 \bigl(\delta^*j_*\pi_0^*\mu_{0*}i^*(E,\theta)^{[n]}\bigr)^{\sym_n}\cong (E,\theta)\,,
\]
which means that we can reconstruct the isomorphism class of $(E,\theta)$ form $(E,\theta)^{[n]}$. 
\end{proof}

\subsection{Stability of tautological Hitchin bundles on Hilbert schemes of points on 
surfaces}\label{subsect:surfacestab}

Let $X$ be a smooth projective surface, and let $H$ be an ample divisor on $X$. There is a unique divisor 
$H_{(n)}$ on the symmetric product $X^{(n)}$ with $\pi_n^*H_{(n)}=\sum_{i=1}^n\pr_i^*H$. We set 
$H_{[n]}:=\mu^*H_{(n)}$. This divisor is big and nef, but not ample, as it is trivial along the exceptional 
divisor of $\mu$. The definitions of stable vector bundles and Hitchin bundles still make sense for non-ample 
divisors. In \cite[Sect.\ 1]{Stapletontaut}, it is shown that, if $E$ is a $H$-slope stable vector bundle on 
$X$ with $E\not\cong \reg_X$, the associated tautological bundle $E^{[n]}$ is $H_{[n]}$-slope stable. Using 
\autoref{lem:C} while going through the proof of \cite{Stapletontaut}, it is quite easy to see that this result 
also generalizes to Hitchin pairs: If $(E,\theta)$ is a $H$-slope stable Hitchin pair on $X$ with $E\not \cong 
\reg_X$, then $(E,\theta)^{[n]}$ is $H_{[n]}$-slope stable.

In \cite{Stapletontaut}, it is shown that in a neighborhood of $H_{[n]}$ in $\NN^1(X)$, there is also an ample 
class $I$ such that $E^{[n]}$ is $I$-slope stable if $E\not\cong \reg_X$ is $H$-slope stable. It seems likely 
that also this result generalizes to Hitchin pairs, but we have not checked the details.

\section*{Acknowledgements}

The first-named author wishes to thank Philipps-Universit\"at Marburg for hospitality while this work
was carried out. He also acknowledges a partial support of a J. C. Bose Fellowship.

\end{document}